\documentclass[11pt]{article}

\usepackage{amsthm}
\usepackage{amssymb}
\usepackage{amsmath,enumerate}
\usepackage{comment}
\usepackage{thm-restate}
\usepackage{url} 
\usepackage{hyperref}
\usepackage{graphicx}
\usepackage{subfigure}
\usepackage[noabbrev,capitalise]{cleveref}
\usepackage[affil-it]{authblk}
\usepackage{color}
\usepackage[normalem]{ulem}

\usepackage[margin=1in]{geometry}

\theoremstyle{plain}
\newtheorem{thm}{Theorem}[section]
\newtheorem{lem}[thm]{Lemma}

\newtheorem{conj}[thm]{Conjecture}

{\noindent \emph{Proof.} {}{#1}{}}{\hfill
	$\Diamond$\vspace{1em}}

\theoremstyle{plain} 
\newcommand{\thistheoremname}{}
\newtheorem{genericthm}[section]{\thistheoremname}

\theoremstyle{definition}

\def\es{\emptyset}
\def\less{\setminus}

\newcommand{\N}{N}

\title{The Erd\H{o}s-Lov\'asz Tihany Conjecture holds for all even-hole-free graphs}
\author{Zi-Xia Song\thanks{Supported by  NSF award DMS-2153945. E-mail address: {\tt Zixia.Song@ucf.edu}.}}
  \affil{ 
  { \small {Department  of Mathematics, University of Central Florida, Orlando, FL 32816, USA}}  
     }

\date{}
\begin{document}
\maketitle
\begin{abstract}
  Let  $s, t\ge2 $ be  integers.  A graph $G$ is \emph{$(s,t)$-splittable} if $V(G)$ can be partitioned into two sets $S$ and $T$ such that $\chi(G[S ]) \ge s$ and $\chi(G[T ]) \ge t$.   The Erd\H{o}s-Lov\'asz Tihany Conjecture from 1968 asserts  that every graph $G$ satisfying  $\omega(G)<\chi(G)=s+t-1$  is $(s,t)$-splittable. A vertex of a graph is \emph{bisimplicial} if the set of its neighbors can be expressed as the union of two cliques.  
Let  $G$ be a graph with   $\omega(G)<\chi(G)=s+t-1$. We prove that if $G$ does not contain $C_4$ as an induced subgraph and every induced subgraph of $G$ has a bisimplicial vertex, then $G$ is $(s,t)$-splittable. Combining our result with a    recent  result of    Chudnovsky and Seymour,  which states that every non-empty even-hole-free graph has a bisimplicial vertex,  we obtain that  the Erd\H{o}s-Lov\'asz Tihany Conjecture holds for all even-hole-free graphs.  \end{abstract}

\baselineskip 18pt

\section{Introduction}
  All graphs in this paper are finite and simple. For a graph $G$, we   use $V(G)$ to denote the vertex set,     $\delta(G)$  the minimum degree,    $\omega(G)$ the clique number and $\chi(G)$ the chromatic number.         For a vertex $x\in V(G)$, we will use $N(x)$ to denote the set of vertices in $G$ which are adjacent to $x$.
We define   $d(x) = |N(x)|$.
If  $A\subseteq V(G)$,    we say that $x$ is \emph{complete} to $A$ if $x$  is adjacent to all vertices in $A$. The subgraph of $G$ induced by $A$, denoted $G[A]$, is the graph with vertex set $A$ and edge set $\{xy \in E(G) \mid x, y \in A\}$. We denote by $V(G) \less A$ the set $V(G) - A$,   and $G \less A$ the subgraph of $G$ induced on $V(G) \less A$. 
If $A = \{a\}$, we simply write   $G \less a$.  
    \medskip

   Our work is motivated by the  celebrated     Erd\H{o}s-Lov\'asz Tihany Conjecture~\cite{Erdos68}. 
Let  $s$ and $ t $ be positive integers. A graph $G$ is \emph{$(s,t)$-splittable} if $V(G)$ can be partitioned into two sets $S$ and $T$ such that $\chi(G[S ]) \ge s$ and $\chi(G[T ]) \ge t$.  In 1968,   Erd\H{o}s~\cite{Erdos68}  published the following conjecture of Lov\'asz, which has since been known as the   Erd\H{o}s-Lov\'asz Tihany Conjecture.

\begin{conj}[The Erd\H{o}s-Lov\'asz Tihany Conjecture]\label{c:ELTC}
Let $s,t\ge 2$ be integers, and let $G$ be a   graph satisfying
 \[\omega(G)<\chi(G)=s+t-1.\]   Then $G$ is $(s,t)$-splittable.  
\end{conj}

 \cref{c:ELTC} is  hard  and few related results are known. 
The case $(2, 2)$ for  \cref{c:ELTC}  is trivial; the cases $(2, 3)$ and   $(3, 3)$  were shown by Brown and Jung~\cite{BJ69} in 1969;
Mozhan~\cite{Moz87} and Stiebitz~\cite{Sti87a} each independently showed the case $(2, 4)$ in 1987; 
the cases   $(3, 4)$ and  $(3, 5)$ were  settled by Stiebitz~\cite{Sti88}  in 1988.  A relaxed version of \cref{c:ELTC} was proved in~\cite{Sti17}.\medskip

Recent work on  \cref{c:ELTC}  have also focused on proving the conjectures for certain classes of graphs. 
A vertex of a graph is  \emph{bisimplicial} if the set of its neighbors is the union of two cliques; a graph is \emph{quasi-line} if   every vertex is bisimplicial. Note that every line graph is  quasi-line and every quasi-line graph is claw-free~\cite{ChSe2012}.  A \emph{hole} in a graph   is an induced cycle of length at least four; a hole is \emph{even} if it has an even length. A graph is \emph{even-hole-free} if it contains no even hole. 
 \cref{c:ELTC}    has  been verified to be true   for line graphs by Kostochka and Stiebitz~\cite{KS08}; 
quasi-line graphs,  and   graphs $G$ with $\alpha(G) = 2$ by Balogh, Kostochka, Prince and Stiebitz~\cite{BKPS09}; graphs  $G$ with $\alpha(G)\ge3$ and    no hole of length between $4$ and $2\alpha(G)-1$ by   the present author~\cite{Song19}. \cref{c:ELTC} remains open for claw-free graphs. Several partial results have been obtained; see, for example, \cite{CFP13,LT26}. We refer the reader to a
recent survey by the present author~\cite{SongSurvey} for further background.   \medskip

  Chudnovsky and Seymour~\cite{CS23}  recently  proved  a structural result on even-hole-free graphs. 

\begin{thm}[Chudnovsky and Seymour~\cite{CS23}]\label{t:evenholefree}  Let $G$ be a    non-empty even-hole-free graph. Then 
  $G$ has a bisimplicial vertex and
  $\chi(G)\le 2\omega(G)-1$. 
\end{thm}

It remains open whether \cref{c:ELTC} holds for even-hole-free graphs. Using \cref{t:evenholefree}, the present author~\cite{Song22} obtained the following partial result.
  
 \begin{thm}[Song~\cite{Song22}]\label{ehfree}
Let $G$ be an  even-hole-free graph  with $\omega(G)<\chi(G)=s+t-1$, where $t\ge s\ge2$. If $s>\chi(G)/3$, then $G$  is $(s,t)$-splittable. 
\end{thm}

 Here we prove that  \cref{c:ELTC}   holds for all even-hole-free graphs.   A graph is \emph{$C_4$-free} if it does not contain $C_4$ as an induced subgraph. We prove a stronger result. 
 
  \begin{thm}\label{t:main}
Let $s,t\ge 2$ be integers, and let $G$ be a $C_4$-free graph such that every induced subgraph of $G$ has a bisimplicial vertex. If 
  $\omega(G)<\chi(G)=s+t-1$,
then $G$ is $(s,t)$-splittable.   
\end{thm}

Combining \cref{t:main} with  \cref{t:evenholefree} leads to the following. 

 \begin{thm}\label{t:ehfree}
Let $s,t\ge 2$ be integers, and let $G$ be an even-hole-free graph satisfying
\[
  \omega(G)<\chi(G)=s+t-1.
\]
Then $G$ is $(s,t)$-splittable.   
\end{thm}

\section{Proof of Theorem~\ref{t:main}}~\label{s:Main}
The proof of \cref{t:main} relies on the following two key lemmas.  For a nonempty set $A\subseteq V(G)$, its \emph{common neighborhood} is defined to be 
\[
  \N_c(A):=\{v\in V(G)\setminus A \mid v \text{ is complete to  } A\}.
\]
Following the definition in \cite{CFP13},  we say that a  clique $K$ in $G$ is    \emph{Tihany} if
\[
  \chi(G\less K)\ge \chi(G)-|K|+1.
\]
It is worth noting that when $\chi(G)=s+t-1$, a Tihany clique of order $s$ in $G$ immediately implies that 
 $G$ is $(s,t)$-splittable.

  \begin{lem}\label{l:Tihanyclique}
Let $G$ be a graph with $\omega(G)<\chi(G)$, and let $K\ne\es$ be a clique in $G$.
If $\N_c(K)$ is a clique, then $K$ is  Tihany.
 
\end{lem}

\begin{proof}
Let   $r:=\chi(G)-|K|$. Since $K$ is a clique and
$\omega(G)<\chi(G)$, we see that  $r\ge 1$ and $\chi(G)\le |K| +\chi(G\less K)$.  It follows that  $\chi(G\less K)\ge r$. Suppose $K$ is not Tihany. Then $\chi(G\less K)\le r $.  Thus $\chi(G\less K)=r$ and   $\chi(G)= |K| +r$.   Let $V_1,\ldots,V_r$ be the color classes of a proper
$r$-coloring of $G\less K$. We claim that for each $i\in\{1,\ldots, r\}$, some vertex in  $V_i$  is complete to $K$. Suppose not. We may assume that no vertex in 
  $V_1$ is complete to $K$. Then $\chi(G[K\cup V_1])=|K|$ because 
  every vertex in $  V_1$ is not adjacent to some vertex   in $K$ and so any proper coloring of $G[K]$ can be extended to $G[K\cup V_1]$.  Thus  $\chi(G)\le  |K|+(r-1)$, a contradiction. \medskip

For each $i\in \{1,\ldots, r\}$, let 
$v_i\in V_i\cap \N_c(K)$. Then vertices $v_1,\ldots,v_r$ are pairwise distinct.  Since $\N_c(K)$ is a clique, we see that $\{v_1,\ldots,v_r\}$ is a clique.  But then $K\cup\{v_1,\ldots,v_r\}$ is a clique of order $|K|+r=\chi(G)$,   which contradicts to the assumption that 
$\omega(G)<\chi(G)$.  This proves that  $\chi(G\less K)\ge r+1$ and so $K$ is Tihany, as desired.
\end{proof}

Note that if a vertex $v$ in a graph $G$ is bisimplicial, then $N(v)$ can be expressed as the union of two disjoint cliques.  

\begin{lem}\label{l:Tihany}
Let $G$ be a $C_4$-free graph, let $v\in V(G)$ such that $N(v)$ is the union of two disjoint cliques, say $A$ and $B$. 
 If $1\le r\le |A|$, then there is a clique 
$K\subseteq A\cup\{v\}$ with $|K|=r+1$ and $v\in K$ such that   $K$ is Tihany\end{lem}

\begin{proof}
For each $a\in A$, let  $\N_B(a):=N(a)\cap B$. Let $q:=|A|$.  We first show that the vertices in $A$ can be enumerated as $a_1, \ldots, a_q$ such that 
\[\N_B(a_1)\subseteq \N_B(a_2)\subseteq\cdots\subseteq \N_B(a_q).\]
 Suppose not. Then there exist   vertices  $x,y\in A$ and vertices $z\in \N_B(x)\less \N_B(y)$ and $w\in \N_B(y)\less \N_B(x)$. 
 Since $A$ and $B$ are cliques, we see that   $G[\{x,y,w,z\}]=C_4$, a contradiction. 
 Recall that $1\le r\le q$. Let $K:=\{v, a_1,\ldots,a_r\}$. Then  $K$ is a clique. We next show that $K$ is Tihany. By \cref{l:Tihanyclique}, it suffices to show that $\N_G(K)$ is a clique. \medskip

For each vertex $u\in \N_c(K)$, we see that $uv\in E(G)$ and so $u\in A\cup B$.    Moreover,
\[
  \N_c(K)
  = (A\setminus K)\cup\bigcap_{i=1}^{r}\N_B(a_i)
  = (A\setminus K)\cup \N_B(a_1).
\]
Both $A\setminus K=\{a_{r+1}, \ldots, a_q\}$ and $\N_B(a_1)\subseteq B$   are cliques. If $j>r$, then
$\N_B(a_1)\subseteq \N_B(a_j)$, and so $a_j$ is adjacent to every vertex in  $\N_B(a_1)$.
It follows that  $\N_c(K)$ is a clique, as desired.
\end{proof}

  We are now ready to prove \cref{t:main}.

\begin{proof}[Proof of Theorem~\ref{t:main}]
 We may assume that  $s\le t$. Let $k: =\chi(G)=s+t-1$.
 We may further assume that $G$ is $k$-critical, that is,  for every $v\in V(G)$, $\chi(G\less v)=k-1$. This is possible by choosing  an induced subgraph $H$ of $G$ that is minimal subject to
$\chi(H)=k$.    Thus  $\delta(G)\ge k-1$. 
 By assumption, $G$ has a bisimplicial vertex $v$. There exist disjoint cliques $A, B\subseteq V(G)$ such that $N(v)=A\cup B$. Recall that $s\le t$.  We may assume that $|A|\le |B|$. Then 
 \[
  2|A|\le |A|+|B|=d(v)\ge\delta(G)\ge k-1=s+t-2\ge 2s-2.
\]
  It follows that $|A|\ge s-1$. Since $G$ is $C_4$-free, by applying  
Lemma~\ref{l:Tihany} to $G$, $v$, and $r=s-1$, there exists a clique $K\subseteq A\cup\{v\}$ with $|K|=s$  such that $K$ is Tihany. Hence  
$\chi(G\less K)\ge \chi(G)-|K|+1=t$. Let $S:=K$ and $T:=V(G)\less K$. 
   Then $\chi(G[S])=s$ and $\chi(G[T])\ge t$ and so $G$ is   $(s,t)$-splittable.
\end{proof}


\begin{thebibliography}{BKPS09}

\bibitem[BJ69]{BJ69}
W.~G. Brown and H.~A. Jung.
\newblock On odd circuits in chromatic graphs.
\newblock {\em Acta Math. Acad. Sci. Hungar.}, 20:129--134, 1969.

\bibitem[BKPS09]{BKPS09}
J\'{o}zsef Balogh, Alexandr~V. Kostochka, Noah Prince, and Michael Stiebitz.
\newblock The Erd\H{o}s-{L}ov\'{a}sz {T}ihany conjecture for quasi-line graphs.
\newblock {\em Discrete Math.}, 309(12):3985--3991, 2009.

\bibitem[CFP]{CFP13}
Maria Chudnovsky, Alexandra Fradkin, and Matthieu Plumettaz.
\newblock On the Erd\H{o}s-{L}ov\'asz {T}ihany {C}onjecture for claw-free
  graphs.
\newblock arXiv:1309.1020.

\bibitem[CS12]{ChSe2012}
Maria Chudnovsky and Paul Seymour.
\newblock Claw-free graphs. {VII}. {Q}uasi-line graphs.
\newblock {\em J. Combin. Theory Ser. B}, 102(6):1267--1294, 2012.

\bibitem[CS23]{CS23}
Maria Chudnovsky and Paul Seymour.
\newblock Even-hole-free graphs still have bisimplicial vertices.
\newblock {\em J. Combin. Theory Ser. B}, 161:331--381, 2023.

\bibitem[Erd68]{Erdos68}
P.~Erd\H{o}s.
\newblock Problems.
\newblock In {\em Theory of {G}raphs ({P}roc. {C}olloq., {T}ihany, 1966)},
  pages 361--362. Academic Press, New York, 1968.

\bibitem[KS08]{KS08}
Alexandr~V. Kostochka and Michael Stiebitz.
\newblock Partitions and edge colourings of multigraphs.
\newblock {\em Electron. J. Combin.}, 15(1):Note 25, 4, 2008.

\bibitem[LT26]{LT26}
Sean Longbrake and Juvaria Tariq.
\newblock Some cases of the Erd\H{o}s-Lovász Tihany Conjecture for claw-free
  graphs.
\newblock {\em Discrete Mathematics}, 349(3):114770, 2026.

\bibitem[Moz87]{Moz87}
N.~N. Mozhan.
\newblock Twice critical graphs with chromatic number five.
\newblock {\em Metody Diskret. Analiz.}, (46):50--59, 73, 1987.

\bibitem[Song19]{Song19}
Zi-Xia Song.
\newblock The Erd\H{o}s-{L}ov\'{a}sz {T}ihany Conjecture for graphs with forbidden
  holes.
\newblock {\em Discrete Math.}, 342(9):2632--2635, 2019.

\bibitem[Song22a]{Song22}
 Zi-Xia Song.
\newblock Some remarks on even-hole-free graphs.
\newblock {\em Electron. J. Combin.}, 29(3):Paper No. 3.30, 9, 2022.

\bibitem[Song22b]{SongSurvey}
 Zi-Xia Song.
\newblock A survey on the {E}rd\H{o}s-{L}ov\'asz {T}ihany Conjecture.
\newblock {\em Adv. Math. (China)}, 51(2):259--274, 2022.

\bibitem[Sti87]{Sti87a}
Michael Stiebitz.
\newblock {$K_5$} is the only double-critical {$5$}-chromatic graph.
\newblock {\em Discrete Math.}, 64(1):91--93, 1987.

\bibitem[Sti88]{Sti88}
Michael Stiebitz.
\newblock On {$k$}-critical {$n$}-chromatic graphs.
\newblock In {\em Combinatorics ({E}ger, 1987)}, volume~52 of {\em Colloq.
  Math. Soc. J\'{a}nos Bolyai}, pages 509--514. North-Holland, Amsterdam, 1988.

\bibitem[Sti17]{Sti17}
Michael Stiebitz.
\newblock A relaxed version of the {E}rd\H{o}s-{L}ov\'{a}sz {T}ihany conjecture.
\newblock {\em J. Graph Theory}, 85(1):278--287, 2017.

\end{thebibliography}
\end{document}